\documentclass[a4paper,12pt,reqno]{amsart}
\usepackage{a4wide}
\usepackage{amsmath}
\usepackage{amssymb}
\usepackage{amsthm}
\usepackage{latexsym}
\usepackage{graphicx}
\usepackage{tikz}
\usepackage[english]{babel}
\usepackage{calc}
\usepackage{accents}

\usetikzlibrary{arrows}

\usepackage[all]{xy}

\def\frk{\frak}               

\def\Phi{{\frk n}}
\def\Phi{{\frk N}}
\def\opn#1#2{\def#1{\operatorname{#2}}} 
\opn\chara{char} \opn\length{\ell} \opn\pd{pd} \opn\rk{rk}
\opn\projdim{proj\,dim} \opn\injdim{inj\,dim} \opn\rank{rank}
\opn\depth{depth} \opn\grade{grade} \opn\height{height}
\opn\embdim{emb\,dim} \opn\codim{codim}

\opn\Tr{Tr} \opn\bigrank{big\,rank}
\opn\superheight{superheight}\opn\lcm{lcm}
\opn\trdeg{tr\,deg}
\opn\reg{reg} \opn\lreg{lreg} \opn\ini{in} \opn\lpd{lpd}
\opn\size{size}\opn\bigsize{bigsize}
\opn\cosize{cosize}\opn\bigcosize{bigcosize}
\opn\sdepth{sdepth}\opn\sreg{sreg}
\opn\link{link}\opn\fdepth{fdepth}
\opn\index{index}
\opn\index{index}
\opn\indeg{indeg}
\opn\N{N}
\opn\SSC{SSC}
\opn\SC{SC}
\opn\conv{conv}
\opn\div{div} \opn\Div{Div} \opn\cl{cl} \opn\Cl{Cl}
\opn\Spec{Spec} \opn\Supp{Supp} \opn\supp{supp} \opn\Sing{Sing}
\opn\Ass{Ass} \opn\Min{Min}\opn\Mon{Mon} \opn\dstab{dstab} \opn\astab{astab}
\opn\Syz{Syz}
\opn\reg{reg}
\opn\Ann{Ann} \opn\Rad{Rad} \opn\Soc{Soc}
\opn\Im{Im} \opn\Ker{Ker} \opn\Coker{Coker} \opn\Am{Am}
\opn\Hom{Hom} \opn\Tor{Tor} \opn\Ext{Ext} \opn\End{End}
\opn\Aut{Aut} \opn\id{id}

\opn\nat{nat}
\opn\pff{pf}
\opn\Pf{Pf} \opn\GL{GL} \opn\SL{SL} \opn\mod{mod} \opn\ord{ord}
\opn\Gin{Gin} \opn\Hilb{Hilb}\opn\sort{sort}
\opn\initial{init}
\opn\ende{end}
\opn\height{height}
\opn\type{type}
\opn\aff{aff} \opn\con{conv} \opn\relint{relint} \opn\st{st}
\opn\lk{lk} \opn\cn{cn} \opn\core{core} \opn\vol{vol}
\opn\link{link} \opn\star{star}\opn\lex{lex}\opn\Mon{Mon}\opn\Min{Min}
\opn\gr{gr}

\def\pot#1#2{#1[\kern-0.28ex[#2]\kern-0.28ex]}

\opn\dirlim{\underrightarrow{\lim}}
\opn\inivlim{\underleftarrow{\lim}}
\let\to=\rightarrow

\def\Implies{\ifmmode\Longrightarrow \else
        \unskip${}\Longrightarrow{}$\ignorespaces\fi}
\def\implies{\ifmmode\Rightarrow \else
        \unskip${}\Rightarrow{}$\ignorespaces\fi}
\def\iff{\ifmmode\Longleftrightarrow \else
        \unskip${}\Longleftrightarrow{}$\ignorespaces\fi}

\let\:=\colon
\newtheorem{Theorem}{Theorem}[section]
 \newtheorem{Lemma}[Theorem]{Lemma}
 \newtheorem{Corollary}[Theorem]{Corollary}
 \newtheorem{Proposition}[Theorem]{Proposition}

 \newtheorem{Example}[Theorem]{Example}
 
 \newtheorem{Definition}[Theorem]{Definition}
 \newtheorem*{Definition*}{Definition}

 \newtheorem*{Conjecture*}{Conjecture}

\let\epsilon\varepsilon
\let\kappa=\varkappa
\def\qed{\ifhmode\textqed\fi
      \ifmmode\ifinner\quad\qedsymbol\else\dispqed\fi\fi}
\def\textqed{\unskip\nobreak\penalty50
       \hskip2em\hbox{}\nobreak\hfil\qedsymbol
       \parfillskip=0pt \finalhyphendemerits=0}
\def\dispqed{\rlap{\qquad\qedsymbol}}

\opn\dis{dis}
\def\pnt{{\raise0.5mm\hbox{\large\bf.}}}

\opn\Lex{Lex}

\begin{document}

 \title{The chain algebra of a pure poset}
\author{Dancheng Lu}

\address{School  of Mathematical Sciences, Soochow University, 215006 Suzhou, P.R.China}
\email{ludancheng@suda.edu.cn}

 \begin{abstract} We extend the notion of  chain algebra, originally defined in  \cite{GN} for finite distributive lattices,  to that of finite pure posets. We show  this algebra is the Ehrhart ring of a (0,1)-polytope, termed the chain polytope, and characterize the indecomposability of this polytope. Furthermore, we prove the normality of the chain algebra, describe its canonical module, and extend one of main results from \cite{GN} by computing its Krull dimension. For width-2 pure posets, we determine the algebra's regularity and conditions for it to be Gorenstein or nearly Gorenstein.
 \end{abstract}

\subjclass[2010]{Primary 13A02, 05E40  Secondary 13H60.}
\keywords{chain algebra, chain polytope, chain semigroup, pure poset,  Canonical module, regularity, Gorenstein, nearly Gorenstein}

 \maketitle

\section{Introduction}
The field dedicated to studying toric rings defined by combinatorial objects is a vibrant area of research, fueled by the dynamic interplay among toric geometry, combinatorics, and commutative algebra. Illustrative examples of such algebras include edge rings of graphs \cite{V, MOT}, base rings of matroids or discrete polymatroids  \cite{B,W, HH2002,L}, toric rings and Ehrhart rings of lattice polytopes \cite{HKMM},  Hibi rings of posets \cite{Hibi2},  rings of stable set polytopes \cite{MOS}, and cycle algebras of matroids \cite{RS2018,RS2024}.

Recently, a novel class of toric rings, called the {\it chain algebra} of a finite distributive lattice, was introduced and studied in \cite{GN}. Let $L$ be a finite distributive lattice and $\mathbb{K}[C_L]$ denote its chain algebra. It was proved in \cite{GN} that the toric ideal of $\mathbb{K}[C_L]$ has a square-free initial ideal, which implies that it is a normal domain. Moreover, the Krull dimension of $\mathbb{K}[C_L]$ is equal to $|L| - \mathrm{rank}(L)$, where $\mathrm{rank}(L)$ is one less than the size of a maximal chain in $L$. Additionally, the conditions under which the toric ideal of $\mathbb{K}[C_L]$ is generated by quadratic binomials were characterized.

  These investigations and results  inspire the following questions:

 (1) Could we describe the canonical module of  a chain algebra?

 (2)  Could we relate a chain algebra to a polytope, given that polytopes are a powerful tool to study algebraic structures?

 (3) Could we define the concept of  chain algebra in a more general setting?

The third question arises from the recognition that the only essential condition in the definition of a chain algebra is that all maximal chains must have the same length, while finite distributive lattices represent a very special type of pure posets. A poset is called  {\it pure} if all its maximal chains have the same length. In this paper, we extend the concept of the chain algebra from finite distributive lattices to pure posets.  The formal definitions are given below.

\begin{Definition}\label{cp} \em   Let $P$  be a finite pure poset, and  let $R := \mathbb{K}[x\:\; x\in P]$ be the polynomial ring over a field $\mathbb{K}$. Each maximal chain of $P$ is then associated with a square-free monomial in $R$, which is the product of all the elements in this chain. The chain algebra of $P$, denoted by $\mathbb{K}[C_P]$, is defined as the $\mathbb{K}$-subalgebra of $R$ generated by  the  monomials associated with all the maximal chains of $P$.

Additionally, we use $C_P$ to denote the set of all functions $f$ from $P$ to $\mathbb{N}_0:=\{0,1,2,\cdots\}$ such that the monomial $U_f:=\prod_{x\in P}x^{f(x)}$ is an element of the chain algebra $\mathbb{K}[C_P]$. It is not difficult to see that $C_P$ is an affine sub-semigroup of $\mathbb{N}_0^P$. We call $C_P$ the {\it chain semigroup} of $P$. 
\end{Definition} Here and hereafter, for sets $A$ and $B$, the notation $A^B$ denotes the set of all maps from $B$ to $A$. Consequently, $\mathbb{N}_0^P$, equipped with a natural addition operation,  is isomorphic to $\mathbb{N}_0^{|P|}$, as semigroups. It is clear that  the chain algebra $\mathbb{K}[C_P]$ is the affine semigroup ring of $C_P$. The objective of this paper is to investigate the algebraic properties of the chain algebra and chain semigroup, with a particular focus on their relationship to the combinatorial characteristics of the poset $P$.

  Our first concern is to depict the functions within $C_P$. To this end, for every pure poset \(P\), we construct a sequence of bipartite graphs, and for an arbitrary function \(f: P \to \mathbb{N}_0\), we introduce the concept of the \( f \)-parallelization of these graphs, which generalizes the standard graph-theoretic parallelization. We prove that \(f \in C_P\) if and only if each graph in the \(f \)-parallelization admits a perfect matching. By applying the Marriage Lemma to this equivalence, we derive a system of linear inequalities that precisely characterizes the functions in \(C_P\), as detailed in Theorem~\ref{main1}. As a  consequence of this result, we show that $C_P$ is a normal affine semigroup and $\mathbb{K}[C_P]$ is a normal domain.

The chain polytope of a pure poset $P$, denoted as $D_P$, is defined as the convex hull in $\mathbb{R}^P$ of the characteristic functions of all maximal chains of $P$. We establish that $D_P$ has the integer decomposition property and prove in Theorem~\ref{ind} that $D_P$ is indecomposable if and only if $P$ cannot be written as the ordinal sum of two non-empty sub-posets.

According to \cite[Theorem 6.3.5]{BH}, the canonical module $\omega_{\mathbb{K}[C_P]}$ of $\mathbb{K}[C_P]$ is isomorphic to an ideal of $\mathbb{K}[C_P]$. As a $\mathbb{K}$-linear space, this ideal is spanned by an ideal of the semigroup $C_P$, which we denote by $K_P$. In Theorem~\ref{can}, we use a system of linear inequalities to characterize the functions in $K_P$.

As per \cite[Proposition 3.1]{HHO2018}, the dimension of the chain algebra $\mathbb{K}[C_P]$ equals the dimension of the $\mathbb{Q}$-linear space spanned by the functions in $C_P$. We show that this space is the solution space of a homogeneous linear system with a row staircase coefficient matrix. This enables us  in Theorem~\ref{krull} to derive a concise formula for the dimension of a chain algebra, generalizing previous results for finite distributive lattices. Significantly, the characterization from Theorem~\ref{can} is essential in the proof of this result.

For pure posets of width 2, we further investigate the algebraic properties of their chain algebras.  By decomposing these posets into ordinal sums of basic posets and anti-chains, the regularity of $\mathbb{K}[C_P]$ is computed, and conditions for $\mathbb{K}[C_P]$ to be Gorenstein or nearly Gorenstein are established. These results highlight the complexity of studying these properties in the general case.

The paper unfolds as follows. Section 2 surveys the background information. Section 3 analyzes the functions from \(C_P\) and proves \(\mathbb{K}[C_P]\) to be a normal domain. Section 4 demonstrates that \(D_P\) has the integer decomposition property and determines when \(D_P\) is an indecomposable polytope. Section 5 details the canonical module of \(\mathbb{K}[C_P]\), while Section 6 derives the Krull dimension formula for \(\mathbb{K}[C_P]\). The final section focuses on the chain algebras of pure posets with width 2.

\section{Preliminaries}
We use $\mathbb{K}$ to denote a field,   $\mathbb{Z}$ and $\mathbb{Q}$ to represent the ring of integers and the rational field, $\mathbb{N}$ to denote the set of positive integers, and $\mathbb{N}_0:=\mathbb{N}\cup\{0\}$ to represent the set of all natural numbers.
In this section, we collect  concepts and basic facts that are  crucial for this paper.

\subsection{Marriage Lemma and perfect matching}

Let $G$ be a simple graph with vertex set $V(G)$ and edge set $E(G)$. By definition, a {\it perfect matching}  of $G$ is a subset of $E(G)$ consisting of pairwise disjoint edges such that their union covers all vertices of $G$.

 A simple graph $G$ is called a {\it bipartite graph} (with partition $(V_1,V_2)$) if $V(G)$ has a decomposition $V(G)=V_1\sqcup V_2$ such that for any edge $e\in E(G)$, $|e\cap V_1|=|e\cap V_2|=1$.  When $G$ is a bipartite graph with  partition $(V_1, V_2)$, we equivalently define  a perfect matching of $G$ as a bijection $\phi$ from $V_1$ to $V_2$ such that each $a\in V_1$ is adjacent to  its image $\phi(a)\in V_2$. For a subset $S\subseteq V(G)$, the {\it neighborhood set} $N_G(S)$ of $S$  in $G$ is the set of vertices of $G$ that are adjacent to some vertex from $S$.  The following lemma, known as the the Marriage Lemma, has several versions. We adapt the version from \cite[Lemma 9.1.2]{HH}.

\begin{Lemma} \label{marriage} Let $G$  be a bipartite graph with partition $(V_1,V_2)$. Then $G$ has a perfect matching if and only if $|V_1|=|V_2|$ and $|S|\leq |N_G(S)|$ for all subset $S\subseteq V_1$.
\end{Lemma}

\subsection{Posets and chain algebras} A partially ordered set (also known as a poset) is a set together with a binary relation $\leq$ that is reflexive, antisymmetric, and transitive. If $x\leq y$ and $x\neq y$, we write $x<y$.

Let $P$ be a finite poset. A {\it chain} of $P$ is a subposet of $P$ such that  every two elements in it are comparable. The {\it length} of a chain $C$ is $\ell(C):=|C|-1$. A chain of $P$ is {\it maximal} if it is not strictly contained in any other chain of $P$. Conversely, an {\it anti-chain} of $P$ is a subset of $P$ in which  every two distinct elements  are incomparable.

\begin{Definition} \em  A finite poset $P$ is called {\it pure} or {\it graded} of rank $r$ for some $r\in \mathbb{N}_0$ if every
maximal chain of $P$ has length $r$.
\end{Definition}

Let $x,y\in P$ with $x< y$. We say $y$ {\it covers} $x$ if there is no $z\in P$ such that $x<z<y$.
For a poset $P$,  a function $\rho: P \rightarrow  \mathbb{N}_0$ is called a {\it rank function} if $\rho$ satisfies the
following conditions:

(1) $\rho(m) = 0$ for every minimal element $m\in P$, and

(2) $\rho(y) = \rho(x) + 1$ for all  $x, y\in  P$ such that $y$ covers $a$.

Every finite pure poset admits a unique rank function.  In what follows, we always assume that $P$ is a finite pure poset of rank $r$ with a rank function $\rho$. For $0\leq i\leq r$, let $P_i$ denote the subset of $P$ consisting of all elements of rank $i$, i.e., $P_i:=\{a\in P\:\; \rho(a)=i\}$. For $i=0,\ldots,r-1$, we define $G_i(P)$ as the bipartite graph with  partition $(P_i,P_{i+1})$ in which for any $a\in P_i$ and $b\in P_{i+1}$, $a$ is adjacent to $b$ if and only if $b$ covers $a$.

  Let $P$ be a finite pure poset and let $R$  denote the polynomial ring $\mathbb{K}[x\:\; x\in P]$. To every maximal chain
$C: x_0<x_1<\cdots<x_r$ in $P$, we associate  a square-free monomial: $x_C:=x_0x_1\cdots x_r$ in $R$. By Definition~\ref{cp}, the  chain algebra $\mathbb{K}[C_P]$ is the $\mathbb{K}$-subalgebra of $R$ generated by  $x_C$ for all the maximal chains $C$ of $P$. Moreover, the chain semigroup $C_P$ has a unique Hilbert basis that consists of characteristic functions of  the maximal chains  of $P$. Here, given a semigroup  $S$, we say that $S$ is generated by a subset $A\subseteq S$ if every element of $S$ is an $\mathbb{N}_0$-linear combination of elements of $A$. Any minimal generating set of $S$ is called a {\it Hilbert basis} of $S$.

 Notice that for any $x\in P_i$, there is a maximal chain of the form: $x_0<x_1\cdots<x_{i-1}<x_i=x<x_{i+1}<\cdots<x_r.$ Consequently,  we conclude that for any $x\in P$,  there is at least one $f\in C_P$ such that $f(x)>0$.

\subsection{Chain polytopes and their algebras} A polytope $P\subset \mathbb{R}^d$ is the convex hull of finite vectors in $\mathbb{R}^d$. The {\it dimension} of a polytope is defined as the dimension of its affine hull. A lattice polytope is a polytope in which all the vertices have integer coordinates, and a (0,1)-polytope is a special type of lattice polytope where all the vertices have coordinates that are either 0 or 1. In our exploration, polytopes emerge as a valuable tool.

\begin{Definition}\em \label{D_P} Let $P$ be a finite pure poset. For each  maximal chain $C$ of $P$, the {\it characteristic function} $\chi_C$ is the function from $P$ to $\mathbb{R}$  such that $\chi_C(x)=1$ if $x\in C$ and $\chi_C(x)=0$ otherwise.   We then define {\it chain polytope} of $P$, denoted by $D_P$,  as the convex hull in   $\mathbb{R}^{P}$  of all these characteristic functions     $\chi_C$.
\end{Definition}

Let $D\subset \mathbb{R}^d$ be a lattice polytope.  Following \cite{HKMM}, the {\it Ehrhart ring} $A(D)$ of $D$ is defined by $$A(D):=\mathbb{K}[x_1^{a_1}\cdots x_d^{a_d}y^k\:\; (a_1,\ldots,a_d)\in kD\cap \mathbb{Z}^d, k\in \mathbb{N}_0].$$  Additionally,  the {\it toric ring} $\mathbb{K}[D]$ of $D$ is defined to be $$\mathbb{K}[D]:=\mathbb{K}[x_1^{a_1}\cdots x_d^{a_d}y\:\; (a_1,\ldots,a_d)\in D\cap \mathbb{Z}^d].$$
They are both  $\mathbb{K}$-subalgebras of the polynomial ring $\mathbb{K}[x_1,\ldots,x_d,y]$ over a field $\mathbb{K}$. If all vectors  in $D\cap \mathbb{Z}^d$ belongs to an affine  hyperplance, the  variable ``$y$" in  the definition of $\mathbb{K}[D]$ may be removed.   In general, one has $\mathbb{K}[D]\subseteq A(D)$, and  it is not difficult to see that $A(D)=\mathbb{K}[D]$ if and only if $D$ has the {\it integer decomposition property}. Following \cite[Definition 4.3]{HHO2018}, a lattice polytope $D$ is said to have the  integer decomposition property (IDP) if for all integers $k>0$ and all $\alpha\in kD\cap \mathbb{Z}^d$, there exist $\beta_1,\ldots \beta_k\in D\cap \mathbb{Z}^d$ such that $\alpha=\beta_1+\cdots+\beta_k$.

The following result is utilized in numerous literatures, such as \cite{RS2024}, without explicit mention. The special case of this result when $|f_1|=\cdots=|f_t| = 2$ was proved in \cite{OH}. Since we are unable to trace its original source, we explicitly state it here and provide a brief proof.

\begin{Lemma} \label{0-1} Let $f_1,\ldots, f_t$ be $(0,1)$-vectors in $\mathbb{R}^{d}$ such that $|f_1|=\cdots=|f_t|$. Then $$\mathrm{conv}(f_1,\ldots,f_t)\cap \mathbb{Z}^d= \{f_1,\ldots,f_t\}.$$ Here, $\mathrm{conv}(f_1,\ldots,f_t)$ denotes the convex hull of $f_1,\ldots,f_t$, and  for any $f\in \mathbb{R}^d$, $|f|$ denotes the number $f(1)+\cdots+f(d).$
\end{Lemma}
\begin{proof} Let $f\in \mathrm{conv}(f_1,\ldots,f_t)\cap \mathbb{Z}^d$. We may write $f=a_1f_1+\cdots+a_tf_t$, where $a_1+\cdots+a_t=1$ with $a_i\geq 0$ for all $i$. From this expression, it follows that $0\leq f(i)\leq 1$ for all $1\leq i\leq d$ and so $f$ is a $(0,1)$-vector. Moreover, we have $|f|=a_1|f_1|+\cdots+a_t|f_t|=|f_i|$ for each $1\leq i\leq t$.
 Since $a_1+\cdots+a_t=1$, there exists at least one $1\leq i\leq t$ such that $a_i>0$. Without loss of generality, we assume that $a_1>0$. We now claim that if $f(i)=1$ then $f_1(i)=1$ for $i=1,\ldots,d$. In fact, if not, there is $1\leq j
\leq d$ such that $f(j)=1$ but $f_1(j)=0$. This implies $1=f(j)=a_2f_2(j)+\cdots+a_tf_t(j)\leq a_2+\cdots+a_t<1$. This  contradiction establishes the claim. Consequently, we have the inclusion $\{1\leq i\leq d\:\; f(i)=1\}\subseteq \{1\leq i\leq d\:\; f_1(i)=1\}$. On the other hand, it is clear that $|f|=|\{1\leq i\leq d\:\; f(i)=1\}|=|\{1\leq i\leq d\:\; f_1(i)=1\}|=|f_1|$. Hence, we conclude that $f=f_1$, completing the proof.
\end{proof}
 The chain polytope $D_P$ is a (0,1)-polytope. In view of Lemma~\ref{0-1}, the chain algebra $\mathbb{K}[C_P]$ coincides with the toric ring $\mathbb{K}[D_P]$.

 Let $A,B$ be subsets of $\mathbb{\mathbb{R}}^d$. Their {\it Minkowski sum}  is defined as $A+B:=\{x+y\:\; x\in A, y\in B\}$.
 Following \cite{HKMM},  a lattice polytope $D$  is called {\it indecomposable} provided that $D$ is not a singleton and  any representation of $D$  as  a Minkowski sum  $D=D_1+D_2$ of lattice polytopes requires that  either $D_1$ or $D_2$ is a singleton. Recall that the direct product of two polytopes $D\subset \mathbb{R}^d$ and $E\subset \mathbb{R}^e$ is defined by $D\times E:=\{(a,b)\:\; a\in D, b\in E\}\subset \mathbb{R}^{d+e}$.
 In the case when $D$ is a (0,1)-polytope,  \cite[Lemma 30]{HKMM} provides an equivalent condition for indecomposability: if $D$ is a non-singleton $(0,1)$-polytope, then $D$ is indecomposable precisely when  there do not exist  non-empty $(0,1)$-polytopes $D_1$ and $D_2$ such that $D=D_1\times D_2$.

\subsection{Nearly Gorenstein and pseudo-Gorenstein}

Let $R$ be a positively graded $\mathbb{K}$-algebra with maximal graded ideal $\mathfrak{m}$. For an $R$-module $M$, its {\it trace}, denoted $\mathrm{tr}_R(M)$, is the sum of the ideals $\phi(M)$ with $\phi\in \mathrm{Hom}_R(M,R)$.  Assume further that $R$  is a Cohen-Macaulay $\mathbb{K}$-algebra admitting a canonical module $w_R$.  Following \cite{HHS},  $R$ is called {\it nearly Gorenstein} if $\mathrm{tr}_R(w_R)\supseteq \mathfrak{m}$. Following \cite{EHHM}, $R$ is called {\it pseudo-Gorenstein} if $\dim_{\mathbb{K}}(\omega_R)_a=1$, where $a:=\min\{i\:\; (\omega_R)_i\neq 0\}$. Both nearly Gorenstein and pseudo-Gorenstein properties serve as generalizations of the notion of Gorensteinness. Nearly Gorensteinness was studied extensively in many papers, such as \cite{HKMM,HMP,HS,Mi}.

\subsection{Homogenous affine semigroups} \label{affine} Recall from \cite{BH} that an affine semigroup  $S$ is a finitely generated sub-semigroup   of $\mathbb{Z}^d$ for some $d>0$ with $\mathbf{0}\in S$. We only consider positive affine semigroups that are contained in $\mathbb{N}_0^d$.
\begin{Definition} \em Let $S\subset \mathbb{N}_0^d$ be an affine semigroup. It is said to be {\it normal} if for $\alpha,\beta,\gamma\in S$ and $n>0$ satisfying $n\alpha=n\beta+\gamma$,  there is $\gamma'\in S$ such that $\gamma=n\gamma'$.
\end{Definition}

A domain is said to be {\it normal} if it is integral closed in its fractional field. For an affine semigroup  $S\subseteq \mathbb{N}_0^d$, the {\it (affine) semigroup ring} $\mathbb{K}[S]$ is defined to be  the $\mathbb{K}$-subalgebra of the polynomial ring $\mathbb{K}[x_1,\ldots,x_d]$ generated by monomials $X^{\alpha}:=x_1^{\alpha_1}\cdots x_d^{\alpha_q}$ with $\alpha=(\alpha_1,\ldots,\alpha_d)\in S$. It is known from \cite{BH} that  $\mathbb{K}[S]$ is a normal domain  if and only if $S$ is a normal affine semigroup, and in this case, $\mathbb{K}[S]$ is Cohen-Macaulay.

Let $S\subset \mathbb{N}_0^d$ be a normal affine semigroup.  An {\it ideal} of $S$ is a subset $T$ of $S$ such that $a+b\in T$ for any $a\in T$ and $b\in S$. Set $$K_S:=\{\alpha\in S\:\; \forall \beta\in S, \exists n>0 \mbox{ and } \gamma\in S \mbox{ such that } n\alpha=\beta+\gamma\}.$$ Then $K_S$ is an ideal of $S$. Let  $\langle K_S\rangle_{\mathbb{K}}$ denote the $\mathbb{K}$-linear subspace of  $\mathbb{K}[S]$ spanned by $X^\alpha$ with $\alpha\in K_S$. Then it is easy to check that  $\langle K_S\rangle_{\mathbb{K}}$ is an ideal of $\mathbb{K}[S]$.  The following lemma is due to  \cite[(21)]{S}.

\begin{Lemma} \label{can0} Let $S\subset \mathbb{N}_0^d$ be a normal affine semigroup. Then the canonical module $\omega_R$ of the semigroup ring $R:=\mathbb{K}[S]$ is isomorphic to the ideal  $\langle K_S\rangle_{\mathbb{K}}$ of $R$.
\end{Lemma}

The only affine semigroup we study in this paper is $C_P$, where $P$ is a finite pure poset. As shown in Section 3,  $C_P$ is a normal affine semigroup.  We will use $K_P$ to denote  $K_{S}$ if $S=C_P$.
\vspace{2mm}

A nonzero element \( h \in H \) is called a \textit{minimal generator} if whenever it is written as \( h = h_1 + h_2 \) with \( h_1, h_2 \in H \), at least one of \( h_1 \) or \( h_2 \) must be zero.

  Recall that an affine semigroup \(H \subseteq \mathbb{Z}^d\) is termed {\it homogenous} when all of its minimal generators reside within an affine hyperplane. In particular, if there exists a vector \(c \in \mathbb{Q}^d\) such that \((c, h) = 1\) for every minimal generator \(h\) of \(H\), then $H$ is homogenous. Here, \((c, h)\) represents the inner-product \(c(1)h(1)+\cdots + c(d)h(d)\) of \(c\) and \(h\).

In this context, the semigroup ring \(\mathbb{K}[H]\) is endowed with a standard grading. This is achieved by assigning  degree one to all the monomials that correspond to the minimal generators of \(H\). Subsequently, we refer to \(\mathbb{K}[H]\) as a homogeneous (or standard graded) semigroup ring.

If \(P\) is a pure poset of rank \(r\), it is readily apparent that \(C_P\) is homogeneous. In this case, the vector \(c\) is defined as \(c(p)=\frac{1}{r + 1}\) for all \(p\in P\).

For all \(i\geq0\), we define \(H_i:=\{a\in H\mid (c,a)=i\}\). Then, \(H\) can be presented as the disjoint union \(H=\bigsqcup_{i\geq0}H_i\), where \(H_0 = \{\mathbf{0}\}\). Furthermore, for all \(i\geq1\), \(H_i\) can be written as the sum \(H_i=H_1+\cdots + H_1\), where $H_1$ repeats $i$ times.  When \(\alpha\in H_i\), we say that \(\alpha\) has degree \(i\) and denote it as \(\deg(\alpha)=i\).

Returning our focus to chain semigroups, \(C_P\) is a homogeneous affine semigroup, and \((C_P)_1\) acts as the unique Hilbert basis of \(C_P\).

The $a$-invariant  of a standard graded $\mathbb{K}$-algebra $R$, denoted by $a(R)$,  is defined to be the degree of the Hilbert series  of $R$,
see \cite[Def 4.4.4]{BH}. It is known from \cite{GW} that if $R$ is a Cohen-Macaulay algebra admitting a canonical module $\omega_R$, then  $a(R)=-\min\{i\:\; (\omega_{R})_i\neq 0\}$. Hence, if $P$ is a finite pure poset, then $a(\mathbb{K}[C_P])$ is equal to $-\min\{i\:\; (K_P)_i\neq \emptyset\}.$

\subsection{Ordinal sum and Segre product} \label{ordinal}
For $\mathbb{N}_0$-graded $\mathbb{K}$-algebras  $S$ and $T$,  their Segre product $S\#T$ is defined as the $\mathbb{N}_0$-graded $\mathbb{K}$-algebra: $R:=\bigoplus_{i\geq 0} S_i\otimes_{\mathbb{K}} T_i.$  If both $T$ and $S$ are Cohen-Macaulay positively graded $\mathbb{K}$-algebras  admitting canonical modules, then $a(R)=\min\{a(T),a(S)\}$ by \cite[Proposition 2.2]{HMP}.

 Given disjoint posets $Q_1$ and $Q_2$, the ordinal sum $Q_1\oplus Q_2$ of $Q_1$ and $Q_2$ is defined to be the poset with the underlying set $Q_1\sqcup Q_2$, where, for any $x_1,x_2\in Q_1\sqcup Q_2$, $x_1\leq x_2$ if either $x_i\in Q_i$ for $i=1,2$, or $\{x_1,x_2\}\subseteq Q_i$ for some $i\in \{1,2\}$ and $x_1\leq x_2$ in the poset $Q_i$.

  If both $Q_1$ and $Q_2$ are pure posets, then  the ordinal sum $Q_1\oplus Q_2$ is also a pure poset.  Furthermore, the chain algebra $\mathbb{K}[C_{Q_1\oplus Q_2}]$  is isomorphic to the Segre product $\mathbb{K}[C_{Q_1}]\#\mathbb{K}[C_{Q_2}]$.

\section{Chain semigroup $C_P$ and  its Normality}

In this section,  we describe the functions in $C_P$ through a combination of  inequalities and equalities. Building upon this description,  we show that the chain algebra $\mathbb{K}[C_P]$ is a normal domain.

  For a finite pure poset $P$  with rank $r$, recall that $P_i$ denotes the subset of $P$ consisting of elements of rank $i$ for $0\leq i\leq r$.
   For a subset $S\subseteq P_i$, where $0\leq i<r-1$, let $N_P(S)$ represent the subset of $P_{i+1}$ comprising  elements that cover some element in $S$. Given a function  $f\in \mathbb{N}_0^P$ and a subset $\emptyset \neq A\subseteq P$, for simplicity, we set $f(A):=\sum_{x\in A}f(x)$. The main result of this section is as follows.

\begin{Theorem} \label{main1}
Let $P$ be a finite pure poset with rank $r$, and let $f$ be a function from $P$ to $\mathbb{N}_0$. Then the following two conditions are equivalent:
\begin{enumerate}
    \item[$\mathrm{(1)}$] The function $f$ belongs to the chain semigroup $C_P$.
    \item[$\mathrm{(2)}$] The following two statements hold:
        \begin{enumerate}
            \item[$\mathrm{(a)}$] For any $0 \leq i, j \leq r$,
                \[
                f(P_i)=f(P_j), \mbox{ i.e. }  \sum_{x \in P_i} f(x) = \sum_{x \in P_j} f(x),
                \]
            \item[$\mathrm{(b)}$] For any $0 \leq i \leq r-1$ and any non-empty subset $V$ of $P_i$,
                \[
               f(V)\leq f(N_P(V)),   \mbox{ i.e. }  \sum_{x \in V} f(x) \leq \sum_{x \in N_P(V)} f(x).
                \]
        \end{enumerate}
\end{enumerate}
\end{Theorem}

To establish this result, we need some preparations. Let $P$ denote  a finite pure poset, and let $f$ be a function mapping every element of $P$ to $\mathbb{N}_0$.   We define $S^f:=\bigsqcup_{x\in S}\{x^{[1]}, \ldots, x^{[f(x)]}\}$ for any non-empty subset $S\subseteq P$. By convention, if $f(x)=0$, then $\{x^{[1]}, \ldots, x^{[f(x)]}\}$ is understood to be the empty set.

   We now introduce the notation  $G_i(P)^f$ to denote the bipartite graph with bipartition  $(P_i^f, P_{i+1}^f)$, where, for any $x^{[k]}\in P_i^f$ with $1\leq k\leq f(x)$ and any $y^{[\ell]}\in P_{i+1}^f$ with $1\leq \ell\leq f(y)$, $x^{[k]}$ is adjacent to $y^{[\ell]}$ if and only if $y$ covers $x$. Following \cite[77.3]{P}, we call   $G_i(P)^f$
the $f$-{\it parallelization} of $G_i(P)$ for $i=0,\ldots,r-1$. Recall that for $i=0,\ldots,r-1$,  $G_i(P)$ is the bipartite graph with bipartition  $(P_i, P_{i+1})$,  where for $a\in P_i$ and $b\in P_{i+1}$, $a$ is adjacent to $b$ if and only if $b$ covers $a$.

  Note that if $G$ is a bipartite graph with  partition $(V_1, V_2)$,   a perfect matching of $G$  may be  defined as   a bijection $\phi$ from $V_1$ to $V_2$ such that each $a\in V_1$ is adjacent to  its image $\phi(a)\in V_2$.

\begin{Lemma} \label{first} Let $f:P\rightarrow \mathbb{N}_0$ be a function. Then $f$ belongs to $C_P$ if and only if $G_i(P)^{f}$ has a perfect matching for all $i=0,\ldots,r-1$.
\end{Lemma}

\begin{proof} By Definition~\ref{cp}, the function $f$ belongs to $C_P$ if and only if the  monomial $U_f$ defined by $U_f:=\prod_{x\in P}x^{f(x)}$ belongs to $\mathbb{K}[C_P]$. To prove our result, we define a map $\Gamma: P^f\rightarrow P$ such that  $\Gamma(a)=x$ for any $a=x^{[j]}\in P^f$ with $x\in P$ and $1\leq j\leq f(x)$. It is easy to see that $U_f=\prod_{a\in P^f} \Gamma(a)$.

 Suppose that for $i=0,\ldots,r-1$, the graph $G_i(P)^f$ has a perfect matching, which we denote by $\phi_i$. Write that $$P_0^f=\{a_1,a_2,\ldots,a_n\}.$$ Let $\varphi_i$ denotes the composition $\phi_i\phi_{i-1}\cdots \phi_0$ for $i=0,\ldots,r-1$. Then
 \begin{align*}& P_1^f=\{\varphi_0(a_1),\varphi_0(a_2),\ldots,\varphi_0(a_n)\},\\ &\quad\quad \cdots,\\
& P_i^f=\{\varphi_{i-1}(a_1),\varphi_{i-1}(a_2),\ldots, \varphi_{i-1}(a_n)\},\\ &\quad\quad \cdots,\\
& P_r^f=\{\varphi_{r-1}(a_1),\varphi_{r-1}(a_2),\ldots, \varphi_{r-1}(a_n)\}.
\end{align*}
Note that  for any $1\leq j\leq n$,
$$\Gamma(a_j)<\Gamma(\varphi_0(a_j))<\cdots<\Gamma(\varphi_i(a_j))<\cdots<\Gamma(\varphi_{r-1}(a_j))$$ is a maximal chain of $P$. Therefore, $U_f$ is expressed as the product $$\prod_{j=1}^n \Big(\Gamma(a_j)\Gamma(\varphi_0(a_j))\cdots \Gamma(\varphi_i(a_j))\cdots \Gamma(\varphi_{r-1}(a_j))\Big ).$$  This implies $U_f$ belongs to $\mathbb{K}[C_P]$, and so $f\in C_P$.

Conversely, suppose that  $U_f$ belongs to $\mathbb{K}[C_P]$. Then there exist a positive integer $n\geq 1$ and $n$ maximal chains of $P$, say $x_{0,j}<x_{1,j}<\cdots<x_{r,j}$ for $j=1,\ldots,n$ such that $$U_f=\prod_{j=1}^n x_{0,j}x_{1,j}\cdots x_{r,j}.$$ Denote by $K_i$ the {\bf multi-set} $\{x_{i,1},x_{i,2}, \ldots, x_{i,n}\}$ for $i=0,\ldots,r$. Note that  $$U_f=\prod_{i=0}^{r}\prod_{x\in K_i}x=\prod_{i=0}^{r}\prod_{a\in P_i^f}\Gamma(a).$$ Fix $i\in [0,r]$. Since  $\prod_{x\in K_i}x=\prod_{a\in P_i^f}\Gamma(a)$, it follows that each $x\in P_i$ appears $f(x)$ times in $K_i$. Thus, there is a bijection, say $w_i$,  from $K_i$ to $P_i^f$ such that $\Gamma(w_i(x_{i,j}))=x_{i,j}$ for all $j=1,\ldots,n$. For $i=0,\ldots,r-1$, consider the following diagram:
 \begin{equation*}\xymatrix{
  K_i \ar[d]_{w_i} \ar[rr]^(.6){\mathbf{b}_i} && K_{i+1} \ar[d]^{w_{i+1}} \\
   P_i^f  \ar@{-->}[rr]^(.6){\exists !\ \phi_i} && P_{i+1}^f .  }
  \end{equation*}
Here, $\mathbf{b}_i$ denotes the bijection $K_i\rightarrow K_{i+1}$ sending  $x_{i,j}$ to $x_{i+1,j}$ for $j=1,\ldots,n$.  Since $\mathbf{b}_i, w_i$ and $w_{i+1}$ are all bijections, there is a unique bijection $\phi_i: P_i^f\rightarrow P_{i+1}^f$ such that the above diagram is commutative. It remains to show that $\phi_i$ is a perfect matching of $G_i(P)^f$.  Let $a\in P_i^f$. There is a unique $1\leq j\leq n$ such that $x_{i,j}\in K_i$  and  $w_i(x_{i,j})=a$. By the definition of $\phi_i$, $\phi_i(a)=w_{i+1}(x_{i+1,j}).$ Therefore, $\Gamma(a)=x_{i,j}<x_{i+1,j}=\Gamma(\phi_i(a))$, implying $a$ is adjacent to $\phi_i(a)$ in $G_i(P)^f$. Consequently, $\phi_i$ is  a perfect matching of $G_i(P)^f$, as required.
\end{proof}

 We now move to prove Theorem~\ref{main1}.

\begin{proof}[Proof of Theorem~\ref{main1}]  In this proof, the notation $G_i(P)^f$ will be written as $G_i^f$ for short.

 (1)$\Rightarrow $ (2) Fix $i\in [0,r-1]$. By Lemma~\ref{first},  $G_i^f$ is a bipartite graph with partition $(P_i^f, P_{i+1}^f)$ that possesses a perfect matching. From this it follows that $|P_i^f|=|P_{i+1}^f|$ by Lemma~\ref{marriage}. As a result, we can conclude that $f(P_i)=f(P_{i+1})$ since $|P_j^f|=f(P_j)$ holds for all $j=0,\ldots,r$. This establishes the validity of statement (a).

Now, let $V$ be a non-empty subset of $P_i$. We may assume without loss of generality that $f(V)\neq 0$, ensuring that $V^f$ is a non-empty subset of $P_i^f$. In view of Lemma~\ref{marriage}, we deduce that $|V^f|\leq |N_{G_i^f}(V^f)|$. Observing that $N_{G_i^f}(V^f)\subseteq N_P(V)^f$ and that $f(N_P(V))=|N_P(V)^f|$, we arrive at the inequality $f(V)\leq f(N_P(V))$. Therefore, statement (b) is also proven.

 (2)$\Rightarrow$ (1): Fix $i\in [0,r-1]$. According to Lemma~\ref{first}, it suffices to show that $G_i^f$ possesses a perfect matching under conditions (a) and (b). Given that $f(P_i)=f(P_{i+1})$, we deduce that $|P_i^f|=|P_{i+1}^f|$. Consider any nonempty subset $S\subseteq P_i^f$. Setting $V=\Gamma(S)$, then it holds that
\[
N_{G_i^f}(S)=(N_P(V))^f \quad \text{and} \quad S\subseteq V^f.
\]
By virtue of (b), we have $|V^f|=f(V)\leq f(N_P(V))$, whence it follows that
\[
|S|\leq |V^f|=f(V)\leq |(N_P(V))^f|=|N_{G_i^f}(S)|.
\]
Consequently, Lemma~\ref{marriage} implies that $G_i^f$ indeed has a perfect matching.\end{proof}

Now, we are ready to prove that the chain algebra $\mathbb{K}[C_P]$ is normal.

\begin{Corollary}\label{normal}
Let $P$ be a finite pure poset. Then $C_P$ is a normal affine semigroup. Consequently, $\mathbb{K}[C_P]$ is a normal domain.
\end{Corollary}

\begin{proof}
Let $f, g,$ and $h$ be functions belonging to $C_P$ such that $nf = ng + h$ for some $n > 0$. To prove that $C_P$ is normal, it suffices to show that $h = nh'$ for some $h' \in C_P$. Define $\alpha: = g - f$. For any $x \in P$, we have $\alpha(x) = g(x) - f(x)$  is an integer. Given that $n\alpha(x) = h(x)$ and $h(x)$ is non-negative, it follows that $\alpha(x) \geq 0$. This establishes that $\alpha$ is a function from $P$ to $\mathbb{N}_0$. Furthermore, since $\alpha = \frac{h}{n}$ and $h\in C_P$, it holds that  $\alpha(X) \leq \alpha(N_P(X))$  for any non-empty subset $X$ of $P_i$ with  $0\leq i\leq r-1$. Additionally, for all $0 \leq i, j \leq r$, we have $\alpha(P_i) = \alpha(P_j)$. These properties of $\alpha$ imply $\alpha \in C_P$ by Theorem~\ref{main1}.

 Therefore, we  have $h = n\alpha$ and $\alpha\in C_P$,  and thus $C_P$ is a normal semigroup. Finally, the normality of $\mathbb{K}[C_P]$ follows from the normality of $C_P$.
\end{proof}

\section{Chain polytope of $P$}

In this section, we prove the chain polytope of any finite pure poset has the integer decomposition property, and characterize when it is indecomposable.

According to Theorem~\ref{main1}, $C_P$ is a homogenous affine semigroup wherein  $\deg(f)=f(P_i)$ for any $i\in [0,r]$. This implies that any $f\in (C_P)_k$ is the sum of $k$ functions from $(C_P)_1$, where $k$ is any integer $>0$. It is worth noting  that $(C_P)_1$ is the set of minimal generators of $C_P$. Here, $(C_P)_k$ is the set of the functions in $C_P$ of degree $k$. For convenience, we call a non-empty subset of $P$  {\it pure} if it is contained in $P_i$ for some $0\leq i\leq r-1$.

\begin{Proposition}\label{IDP}
Let $P$ be a finite pure poset. Then $D_P$  has the integer decomposition property. Consequently, $\mathbb{K}[C_P]$ is the same as the Ehrhart ring $A(D_P)$.
\end{Proposition}
\begin{proof} Let $f\in kD_P\cap \mathbb{Z}^{P}$, where $k$ is any positive integer. To prove that $D_P$ has IDP, it suffices to show that $f$ is contained in $(C_P)_k$. We do this in the following steps.  First, we notice that that $f$ is a function from $P$ to $\mathbb{N}_0$. Next, by Definition~\ref{D_P}, we may write $f=k(a_1f_1+\cdots+a_tf_t)$, where $f_i\in (C_P)_1$ for all $i=1,\ldots,t$ and the coefficients $a_i$ satisfy $a_1+\cdots+a_t=1$ with $a_i\geq 0$ for all $i=1,\ldots,t$.   According to Theorem~\ref{main1}, for  all $i=1,\ldots,t$, we have $f_i(P_j)=1$ for $j=0,\ldots,r$ and $f_i(X)\leq f_i(N_P(X))$ for any pure subset $X$ of $P$. From these, it follows that $f(P_j)=k$ for $j=0,\ldots,r$ and that $$f(X)=k\sum\limits_{j=1}^ta_jf_j(X)\leq k\sum\limits_{j=1}^ta_jf_j(N_P(X))=f(N_P(X))$$ for any pure subset $X$ of $P$. Thus, $f$ belongs to $(C_P)_k$, as required.
\end{proof}

By \cite[Theorem 4.3]{HHO2018}, if a lattice polytope $D$ has  the integer decomposition property, then $\mathbb{K}[D]$ is normal. However, the converse is not true in general, as shown by \cite[Example 4.4]{HHO2018}. Hence, Corollary~\ref{normal} can also be derived from Proposition~\ref{IDP} together with \cite[Theorem 4.3]{HHO2018}.
We now classify when $D_P$ is indecomposable.

\begin{Theorem}\label{ind}
Let $P$ be a finite pure poset with rank $r\geq 1$. Then the following statements are equivalent:
 \begin{enumerate}
\item The chain polytope $D_P$  is indecomposable.
\item $G_i(P)$ is not a complete bipartite graph for all $i=0,\ldots, r-1$.
\item $P$ cannot be expressed as the ordinal sum of two non-empty sub-posets of $P$.
 \end{enumerate}
\end{Theorem}

\begin{proof} $(2)\Rightarrow (3)$ Suppose that $P$ is the ordinal sum of non-empty sub-posets $Q_1$ and $Q_2$ of $P$. Since $P$ is pure, it follows that both $Q_1$ and $Q_2$ are pure. Let $k:=\max\{\rho(a)\:\; a\in Q_1\}$. Then $G_k(P)$ is a complete bipartite graph, a contradiction.

$(3)\Rightarrow (2)$
If $G_i(P)$ is a complete bipartite graph for some $i$, then $P$ is the ordinal sum of $Q_1:=P_0\cup\cdots\cup P_i$ and $Q_2:=P_{i+1}\cup\cdots\cup P_r$, a contradiction.

 If  $P$ is the ordinal sum $Q_1\oplus Q_2$, then $D_P=D_{Q_1}\times D_{Q_2}$. This proves $(1)\Rightarrow (3)$. It remains to prove $(2)\Rightarrow (1)$.

$(2)\Rightarrow (1)$  Suppose on the contrary  that $D_P$  is decomposable. Then, according to \cite[Lemma 30]{HKMM}, there exist $(0,1)$-polytopes $D_1\subset \mathbb{R}^{Q_1}$ and $D_2\subset \mathbb{R}^{Q_2}$ such that $D_P$ is the direct product $D_1\times D_2$. Here, $Q_1$ and $Q_2$ are non-empty subsets of $P$ such that $Q_1\cap Q_2=\emptyset$ and $Q_1\cup Q_2=P$.

We use the {\it support} of $f$ to refer to the subset $\{x\in P\:\; f(x)=1\}$  for any $f\in D_P\cap \mathbb{Z}^P$ or $f\in D_1\cap \mathbb{Z}^{Q_1}$ or $f\in D_2\cap \mathbb{Z}^{Q_2}$.  It is clear that every $f$ is determined  by its support. By $D_P=D_1\times D_2$, it  means that if $S$ is the support of a $(0,1)$-vector in $D_P$, then $S\cap Q_1$ and $S\cap Q_2$ are the supports of $(0,1)$-vectors in $D_1$ and $D_2$ respectively. Conversely, if $S_1$ and $S_2$ are the supports of $(0,1)$-vectors in $D_1$ and $D_2$ respectively, then $S_1\cup S_2$ is the support of a $(0,1)$-vector in $D_P$, i.e., a maximal chain of $P$.

We claim that for any $a_1\in Q_1$ and any $a_2\in Q_2$, $a_1$ and $a_2$ are comparable in $P$. Since $P$ is pure, there exist maximal chains $C_1$ and $C_2$ of $P$ such that $a_j\in C_j$ for $j=1,2$. Since $Q_j\cap C_1$ and $Q_j\cap C_2$ are supports of some (0,1)-vectors of $D_j$ for $j=1,2$, the union of $Q_1\cap C_1$ and $Q_2\cap C_2$ is  a maximal chain of $P$. Given that $a_i\in Q_i\cap C_i$ for $i=1,2$, it follows that $a_1$ and $a_2$ belongs to a  maximal chain of $P$. In particular, $a_1$ and $a_2$ are comparable, proving the claim.

 By  this  claim, we have for each $i=0,\ldots,r$, either $P_i\subseteq Q_1$ or $P_i\subseteq Q_2$. Since both $Q_1$ and $Q_2$ are non-empty, there exists $k\in [0,r-1]$ such that $P_k\subseteq Q_1$ and $P_{k+1}\subseteq Q_2$. However,  the claim  implies that $G_k(P)$ is a complete bipartite graph, a contradiction. \end{proof}
 According to \cite[Theorem 32]{HKMM},  if $D$ is an indecomposable lattice polytope, then the Ehrhart ring $A(D)$ is Gorenstein if and only if it is nearly Gorenstein. The following example illustrates  an application of Theorem~\ref{ind}.
\begin{Example}\em  Let $G$ be a connected bipartite graph. Note that a bipartite graph can be regarded as a pure poset of rank one in a natural way, and its edge algebra is isomorphic to the chain algebra of the corresponding poset.  By Theorem~\ref{ind},  the  edge polytope of $G$ is indecomposable precisely when  $G$ is not complete. Hence, according to \cite[Theorem 32]{HKMM},   the edge ring $\mathbb{K}[G]$ of $G$ is Gorenstein if and only if it is nearly Gorenstein in the case when $G$ is not complete.

\end{Example}

\section{$K_P$ and Canonical modules}

In this section we  provide an explicit description of  the canonical module of $\mathbb{K}[C_P]$ for every pure poset $P$. To do so, the following definition is essential.

\begin{Definition} \em Let $P$ be a finite pure poset. A non-empty subset of $P$ is {\it pure} if it   is contained in $P_i$ for some $i \in \{0, 1, \ldots, r-1\}$. For a pure subset  $X$   of $P$, we call $X$  {\it strict} if there exists $f\in C_P$ such that $f(X)< f(N_P(X))$.
\end{Definition}

Recall that $K_P$ is defined as the ideal of $C_P$ consisting of all functions $f \in C_P$ such that for any $g \in C_P$, there exist a positive integer $n$ and an element $h$ in $C_P$ satisfying the equation $nf = g + h$.
 By Lemma~\ref{can0}, the canonical module of $\mathbb{K}[C_P]$ is isomorphic to $\langle K_P\rangle_{\mathbb{K}}$, the $\mathbb{K}$-linear subspace of  $\mathbb{K}[C_P]$ spanned by $\prod_{x\in P}x^{f(x)}$ with $f\in K_P$. Hence, elucidating the structure of the canonical module of $\mathbb{K}[C_P]$ is equivalent to characterizing the functions contained in $K_P$.

\begin{Lemma} \label{can1} Let $f\in C_P$. Then $f$ belongs to $K_P$ if and only if the following conditions hold:

$\mathrm{(1)}$  $f(X)< f(N_P(X))$ for any strict subset $X$ of $P$;

$\mathrm{(2)}$  $f(x)>0$ for all $x\in P$.
\end{Lemma}

\begin{proof} $\Rightarrow $  Let $f$ be a function belonging to  $K_P$. Consider any strict subset   $X$  of $P$. There exists $g\in C_P$ such that $g(X)< g(N_P(X))$. Given that $f\in K_P$, it follows that there exists $n>0$ and $h\in C_P$ such that $nf=g+h$. As a result, we have $$f(N_P(X))-f(X)=\frac{1}{n}(g(N_P(X))-g(X))+ \frac{1}{n}(h(N_P(X))-h(X))>0.$$

 Take an arbitrary element $x\in P$. Since there is at least one maximal chain of $P$ containing $x$, it follows that there is $g'\in C_P$ such that $g'(x)>0$. Since $nf=g'+h'$ for some $n>0$ and for some $h'\in C_P$, we obtain $nf(x)=g'(x)+h'(x)>0$. From this, we conclude that $f(x)>0$.

$\Leftarrow$  Suppose that $f$ is a function in $C_P$ that fulfills both conditions (1) and (2).  Given   an arbitrary function  $g\in C_P$,   it is routine to verify that the function $h:=nf-g$ belongs to $C_P$, where $n$ is given as $$n:=\max\{g(N_P(X))-g(X)\:\; X \mbox{  is a strict set of } P\}+\max\{g(x)\:\; x\in P\}.$$
  Therefore,  we conclude that $f\in K_P$, as desired.
\end{proof}

To  characterize all the strict subsets of $P$, we introduce several additional notions. 
Fix $i\in \{0,\ldots,r-1\}$.  Notice that the bipartite graph $G_i(P)$ is the disjoint union of its connected components, say $G_{i,1}, \ldots, G_{i,k_i}$. Since all $G_{i,j}$ are connected bipartite graphs,  we may assume that $G_{i,j}$  has a bipartition $(P_{i,j}, N_P(P_{i,j}))$ such that $P_{i,j}\subseteq P_i$ for $j=1,\ldots,k_i$. One easily observes that $P_i=P_{i,1}\sqcup \cdots \sqcup P_{i,k_i}$ and $P_{i+1}=N_P(P_{i,1})\sqcup \cdots \sqcup N_P(P_{i,k_i}).$

\begin{Definition}\label{5.3} \rm
For all $i=0,\ldots,r-1$ and  $j=1,\ldots,k_i$, we define  $P_{i,j}$ as a {\it connected component of rank } $i$.
\end{Definition}

In the following result, we  identify all the pure subsets of $P$ that are strict.

\begin{Lemma} \label{can2}
Let \( P \) be a finite pure poset of rank \( r \).  Fix an integer \( i \) with \( 0 \leq i \leq r-1 \) and consider any pure subset \( X \subseteq P_i \). Then \( X \) is a strict set if and only if there exists at least one \( j \in \{1, \ldots, k_i\} \) such that \( X \cap P_{i,j} \) is a non-empty proper subset of \( P_{i,j} \).
\end{Lemma}

\begin{proof}
We first show that no connected component of rank \( i \) is a strict set. Fix \( f \in C_P \). Then \( f(P_{i,j}) \leq f(N_P(P_{i,j})) \) for all \( j = 1, \ldots, k_i \). Furthermore, we have
\[
f(P_i) = \sum_{j=1}^{k_i} f(P_{i,j}) \quad \text{and} \quad f(P_{i+1}) = \sum_{j=1}^{k_i} f(N_P(P_{i,j})).
\]
Since \( f(P_i) = f(P_{i+1}) \), it follows that \( f(P_{i,j}) = f(N_P(P_{i,j})) \) for all \( j = 1, \ldots, k_i \). Thus, no \( P_{i,j} \) is a strict set.

Next, we show that every non-empty proper subset of any connected component of rank \( i \) is a strict set. Let \( S \) be a non-empty proper subset of \( P_{i,j} \) for some \( j \). Since the bipartite graph \( G_{i,j} \) with bipartition \( (P_{i,j}, N_P(P_{i,j})) \) is connected, there exist \( \alpha \in S \) and \( \beta \in P_{i,j} \setminus S \) such that \( N_P(\alpha) \cap N_P(\beta) \) is non-empty. Let \( \gamma \in N_P(\alpha) \cap N_P(\beta) \). Then there is at least one maximal chain of the following form:
\[
\alpha_0 < \alpha_1 < \alpha_2 < \cdots < \alpha_{i-1} < \beta < \gamma < \gamma_{i+1} < \cdots < \gamma_r.
\]
Define the function \( w: P \to \mathbb{N}_0 \) corresponding to this chain by setting \( w(p) = 1 \) if \( p \) belongs to this chain and \( w(p) = 0 \) otherwise. Since \( \gamma \in N_P(S) \), we have
\[
\sum_{p \in S} w(p) = 0 \quad \text{and} \quad \sum_{p \in N_P(S)} w(p) \geq w(\gamma) = 1.
\]
This implies \( w(S) < w(N_P(S)) \), thereby proving that \( S \) is a strict set.

Finally, let \( X_j = X \cap P_{i,j} \) for $j=1,\ldots,k_i$. Then \( X = \bigsqcup_{j=1}^{k_i} X_j \) and \( N_P(X) = \bigsqcup_{j=1}^{k_i} N_P(X_j) \).  Therefore, \( f(X) < f(N_P(X)) \) if and only if there is at least one \( j \in \{1, \ldots, k_i\} \) such that \( f(X_j) < f(N_P(X_j)) \). The result now follows  from this observation, combined with the two conclusions established earlier.
\end{proof}

We are now in the position  to present the main result of this section.
\begin{Theorem} \label{can} Let $P$  be  a finite pure poset of rank $r$. Then for any  $f\in C_P$,  $f$ belongs to $K_P$ if and only if the following conditions are satisfied:

$\mathrm{(1)}$ For any  $0\leq i\leq r-1$ and for any non-empty proper subset of any connected component of rank $i$, it holds that $f(X)< f(N_P(X))$,

$\mathrm{(2)}$  For all $x\in P$, $f(x)>0$ .
\end{Theorem}
\begin{proof} Let $f\in C_P$. Given Lemmas~\ref{can1} as well as \ref{can2}, we only  need to show that if condition (1) holds, then $f(X)< f(N_P(X))$ for any strict subset $X$ of $P$. Let $X\subseteq P_i$ be a strict subset of $P$. By  Lemma~\ref{can2}, there exists  $1\leq j_0\leq k_i$ such that $X\cap P_{i,j_0}$ is a non-empty proper subset of $P_{i,j_0}$. It follows from condition (1) that $f(X\cap P_{i,j_0})<f(N_P(X\cap P_{i,j_0}))$. Consequently, we have $$f(X)=\sum\limits_{j=1}^{k_i}f(X\cap P_{i,j})<\sum\limits_{j=1}^{k_i}f(N_P(X\cap P_{i,j}))=f(N_P(X)),$$ as required.
\end{proof}

\section{Dimension of a chain algebra}
In this section, we present an explicit formula for the Krull dimension of  chain algebras,  alongside a definitive formula for the dimension of chain polytopes associated with pure posets. We begin with a result that may be   well-known in linear algebra; however, for completeness,  we include a proof.

\begin{Lemma} \label{linear}  Let $V$ be a linear subspace of the $\mathbb{Q}$-space $\mathbb{Q}^n$ and let $S$ be an open subset of $\mathbb{Q}^n$. If $S\cap V\neq \emptyset$, then $V$ is spanned by the vectors belonging to $S\cap V$.
\begin{proof}Let $\alpha_0\in S\cap V$ and let $\alpha_1,\ldots,\alpha_r$ be a $\mathbb{Q}$-basis for $V$. Since $\alpha_0$ is an interior point of $S$, there exists a positive integer $k>0$ such that $\frac{1}{k}(\alpha_i-\alpha_0)+\alpha_0$ belongs to $S$ for $i=1,\ldots,r$.  Thus, $S\cap V$ contains the vectors $\frac{1}{k}(\alpha_i-\alpha_0)+\alpha_0$ for $i=1, \ldots, r$ as well as $\alpha_0$. This implies that $\alpha_1,\ldots,\alpha_r$ belong to $\langle S\cap V\rangle_{\mathbb{Q}}$. Here, $\langle S\cap V\rangle_{\mathbb{Q}}$ denotes the subspace of  $\mathbb{Q}^n$  spanned by $S\cap V$. Hence, we have $V=\langle S\cap V\rangle_{\mathbb{Q}}$, as required.
\end{proof}

\end{Lemma}

We refer to Definition~\ref{5.3} for the definition of connected components of rank $i$.

\begin{Theorem} \label{krull} Let $P$ be a finite pure poset with rank r. Let $k_i$ denote the number of connected components of rank $i$ for $i=0,\ldots, r-1$. Then $$\dim \mathbb{K}[C_P]=|P|-(k_0+\cdots+k_{r-1}).$$
\end{Theorem}

\begin{proof} For each $p\in P$, let $X_p$ denote an indeterminate indexed by $p$. Consider the following systems of equalities or inequalities:
\begin{equation} \tag{\S} \sum_{p\in P_{i,j}} X_p=\sum_{p\in N_P(P_{i,j})}X_p,\qquad \mbox{ for } i=0,\ldots,r-1, j=1,\ldots,k_i.
\end{equation}

\begin{equation} \tag{\dag} X_p\geq 0, \quad \mbox{  for } p\in P.
\end{equation}

\begin{equation} \tag{$\dag_{<}$} X_p> 0, \quad \mbox{  for } p\in P.
\end{equation}

\begin{equation} \tag{\ddag} \sum_{p\in S} X_p\leq \sum_{p\in N_P(S)}X_p ,    \mbox{ for } \emptyset\neq S\subsetneq P_{i,j}, 1\leq i\leq r-1, 1\leq j\leq k_i.
\end{equation}

\begin{equation} \tag{$\ddag_{<}$} \sum_{p\in S} X_p<\sum_{p\in N_P(S)}X_p , \mbox{ for } \emptyset\neq S\subsetneq P_{i,j}, 1\leq i\leq r-1, 1\leq j\leq k_i.
\end{equation}

 It is natural to identify every  rational or integral solution to part or all of these equalities or inequalities with  a function mapping from $P$ to the set of rational numbers or integral numbers.
Under this identification,
 we have \begin{itemize}
  \item  $C_P$ consists of  all integral solutions that satisfy $(\S), (\dag)$ and $(\ddag)$ by Theorem~\ref{main1}, and
   \item $K_P$  consists of  all integral solutions that satisfy $(\S), (\dag_{<})$ and $(\ddag_{<})$ by Theorem~\ref{can}.
   \end{itemize}

For a subset $M$ of the $\mathbb{Q}$-space $\mathbb{Q}^P$ (the set of all functions from $P$ to $\mathbb{Q}$), we denote by $ \langle M \rangle_{\mathbb{Q}}$ the $\mathbb{Q}$-subspace of $\mathbb{Q}^P$ spanned by $M$.  It is evident that the Krull dimension of $\mathbb{K}[C_P]$ is equal to $\dim_{\mathbb{Q}} \langle C_P \rangle_{\mathbb{Q}}$.  We also introduce the following notions:
\begin{itemize}
\item Let $V$ denote the set of all rational solutions to $(\S)$;
\item Let $A$ denote the set of all rational solutions to   $(\dag_{<})$ and $(\ddag_{<})$;
\item Let $B$ denote the set of all rational  solutions to   $(\dag)$ and $(\ddag)$.
\end{itemize}

Note that $V$ is a $\mathbb{Q}$-linear subspace of $\mathbb{Q}^P$, while $A$ is an open subset of $\mathbb{Q}^P$. It is straightforward to observe that $V\cap A\subseteq V\cap B\subseteq V$.  Since $K_P \subseteq V\cap A$, it follows that  $V\cap A$ is non-empty.  Therefore, according to Lemma~\ref{linear}, $\langle V\cap A\rangle_{\mathbb{Q}}$ coincides with $V$. Note that $C_P$ and $V\cap B$   are respectively composed of the integer solutions and the rational solutions of the same system  of homogeneous linear inequalities. Thus, they span the same $\mathbb{Q}$-linear space, that is,  $\langle C_P\rangle_{\mathbb{Q}}=\langle V\cap B\rangle_{\mathbb{Q}}$. It follows that the Krull dimension of $\mathbb{K}[C_P]$, which equals to the dimension of   $\langle C_P\rangle_{\mathbb{Q}}$ by \cite[Proposition 3.1]{HHO2018}, is equal to the dimension $\dim_{\mathbb{Q}}V$. Finally, observing that the rows of the coefficient matrix of the  homogeneous system $(\S)$ of linear equations  are linear independent (in fact, the coefficient matrix of $(\S)$ is in row echelon form), the result follows from the basic linear algebra. \end{proof}

An application of Theorem~\ref{krull} recovers \cite[Theorem 3.4]{GN}.

\begin{Corollary} Let $L$ be   a finite distributive lattice of  rank $r$.  Then the chain algebra $\mathbb{K}[C_L]$ has the dimension $|L|-r$.
\end{Corollary}
\begin{proof} In view of the initial portion of the proof of \cite[Theorem 3.4]{GN}, $G_i(L)$ is a connected bipartite graph for each $0\leq i\leq r-1$. Consequently, the desired result follows  from Theorem~\ref{krull}.
\end{proof}

  In view of \cite[Section 4.2.4]{HHO2018} together with \cite[Theorem 3.1]{HHO2018}, the dimension of a lattice polytope $D$ is equal to the  dimension of the toric ring $\mathbb{K}[D]$ minus one. Hence, the following result is immediate.

\begin{Corollary} With the same assumptions as in Theorem~\ref{krull}, we have $\dim D_P=|P|-(k_0+\cdots+k_{r-1})-1$.
\end{Corollary}

\section{Posets of width 2}
To explore the algebraic properties of chain algebras such as regularity, Gorensteinness, and near-Gorensteinness, we need to refine our focus. A significant, yet simpler, class of pure posets to consider is those of width 2.
Recall the {\it width} of a poset is the maximal size of an anti-chain contained in this poset. Let $P$ be a pure poset of width 2.  In this section, we will compute the regularity of $\mathbb{K}[C_P]$ and characterize the conditions under which $\mathbb{K}[C_P]$ is Gorenstein or nearly Gorenstein.

 Hereafter, we always assume that $P$ is a pure poset of width 2 and  rank $r\geq 1$ unless otherwise stated. As previously, we denote by $P_i$  the set of elements of $P$ with rank $i$, and by $G_i(P)$  the bipartite graph with bipartition $(P_i, P_{i+1})$, where for $a\in P_i$ and $b\in P_{i+1}$, $a$ is adjacent to $b$ if and only if $a<b$. It is clear that $P_i$ contains at most 2 elements for each $i$.  We make the following observations:

 (1) If, for some $i$, $G_i(P)$ is not connected,  then $G_i(P)$   consists of two disjoint edges. We may assume $P_i=\{a_1,a_2\}$ and $P_{i+1}=\{b_1,b_2\}$ with $a_j<b_j$ for $j=1,2$. In this case, we have $f(a_i)=f(b_i)$ for $i=1,2$ and for $f\in C_P$. Thus, if we identify  $a_j$ with the  $b_j$ for $j=1,2$ to obtain a smaller poset $P'$, then it is evident that $C_P$ is isomorphic to $C_{P'}$, and $K_P$ is isomorphic to $K_{P'}$. In particular, we have $\mathbb{K}[C_P]$ is isomorphic to $\mathbb{K}[C_{P'}]$.   Hereafter, we denote by $\overline{P}$ the poset that is obtained from $P$ by a series of such identifications, where $G_i(\overline{P})$ are connected for all possible $i$. For example, if $P$ is the disjoint union of two chains $C_1$ and $C_2$ of the same length, where every element of $C_1$
 is incomparable with every element of $C_2$, then $\overline{P}$ is an anti-chain containing exactly two elements.

 (2) If there is an index $0<i<r$ such that $G_i(P)$ is complete bipartite, then $P$ is the ordinal sum of the subposets $Q_1:=P_0\cup \cdots\cup P_i$ and $Q_2:=P_{i+1}\cup \cdots\cup P_r$. Furthermore, if $|P_i|=1$, then $\mathbb{K}[C_{Q_1}]$ is isomorphic to $\mathbb{K}[C_{Q_1'}]$, where $Q_1'$ denotes the subposet $Q_1\setminus P_i$, and  similarly, if $|P_{i+1}|=1$, then $\mathbb{K}[C_{Q_2}]$ is isomorphic to $\mathbb{K}[C_{Q_2'}]$, where $Q_2':=Q_2\setminus P_{i+1}$.

The following definition is now natural.
 \begin{Definition}\em A finite pure poset $P$ is called {\it basic} if it is of width 2  and rank $\geq 1$, and    $G_i(P)$ is a connected  bipartite graph that is not complete for each $i=0,\ldots,r-1$.
 \end{Definition}
 The above discussions establish that, given any pure poset $P$ with width 2, we have a poset $\overline{P}$ such that $G_i(\overline{P})$ is connected for all possible $i$ and   $\mathbb{K}[C_P]$ is isomorphic to $\mathbb{K}[C_{\overline{P}}]$. Furthermore, $\overline{P}$ is the ordinal sum of basic posets and posets of rank $0$ and width 2. Note that a poset of rank $0$ and width 2 is nothing but an anti-chain containing two elements.

 In what follows, whenever $P$ is a basic poset, we always assume that $P=\{X_0,\ldots,X_r, Y_0,\ldots,Y_r\}$ such that $X_0<X_1<\cdots<X_r$ and $Y_0<\cdots<Y_r$ are two maximal chains of $P$.  For $X\in P$, we say $X$ has degree $i$ if $|N_P(X)|=i$, meaning that there are exactly $i$ elements that covers $X$.
 By our assumption, for each $i\in [0,r-1]$,  exactly one of $\{X_i,Y_i\}$  has degree one. Furthermore, if $X_i$ has  degree  one,  then $N_P(X_i)=\{X_{i+1}\}$ and $N_P(Y_i)=\{X_{i+1},Y_{i+1}\}$, and similarly, if $Y_i$ has degree  one then $N_P(Y_i)=\{Y_{i+1}\}$ and $N_P(X_i)=\{X_{i+1},Y_{i+1}\}$. Thus,  to fully describe a basic poset, it is enough to identify all the elements  that have degree one.
We associate every basic poset a sequence of positive integers as follows.

\begin{Definition}\em Let $(c_1,c_2,\ldots,c_n)$ be an integral vector, where $c_1+\cdots+c_n=r$ with $c_i\geq 1$ for all $i=1,\ldots,n$. Set $b_0=0$ and $b_i=c_1+c_2+\cdots+c_i$ for $i=1,\ldots,n$. Let $P$ be a basic poset. By relabelling if necessary, we may assume $X_0$ has degree one. We say that a basic poset $P$ has  type $(c_1,c_2,\ldots,c_n)$   if the degree-one elements of $P$ are as follows: \begin{align*}& X_0,\ldots,X_{c_1-1}, Y_{c_1},\ldots, Y_{c_1+c_2-1}, \ldots, \\& X_{b_{2i}},X_{b_{2i}+1}, \ldots, X_{b_{2i+1}-1}, Y_{b_{2i+1}},Y_{b_{2i+1}+1}, \ldots, Y_{b_{2i+2}-1}, \ldots\\ & W_{b_{n-1}},W_{b_{n-1}+1},\ldots, W_{b_n-1}.  \end{align*} Here, $W=X$ if $n$ is odd, and $W=Y$ if $n$ is even.
\end{Definition}
 Up to isomorphisms, every basic poset is determined  by its type. Specializations of Theorems~\ref{main1} and \ref{can}  to basic posets produces the following result.

\begin{Lemma} \label{cri} Let $P$ be a basic poset of type $(c_1,\ldots,c_n)$ with $c_0$ set to 0.  Consider a function $f: P\rightarrow \mathbb{N}_0$  such that $f(P_i)=f(P_j)$ for all $0\leq i,j\leq r$. Then

$\mathrm{(1)}$ $f$ belongs to $C_P$ if and only if the following  condition $\mathrm{(*)}$ holds: for all $i = 0, \ldots, n$, when $c_1 + \cdots + c_i  \leq k \leq c_1 + \cdots + c_i + c_{i+1} - 1$:

\begin{itemize}
    \item If $i$ is even, then $f(X_k) \leq f(X_{k+1})$;
    \item If $i$ is odd, then $f(X_k) \geq f(X_{k+1})$.
\end{itemize} More intuitively, condition $\mathrm{(*)}$ can also be written in the following form:
$$f(X_0)\leq\cdots \leq f(X_{c_1})\geq f(X_{c_1+1})\geq \cdots \geq f(X_{c_1+c_2})\leq \cdots \geq \cdots f(X_{r-1})\sim f(X_r).$$
Here, $\sim$ is $\leq $ if $n$ is odd and $\sim$ is $\geq $ if $n$ is even.

$\mathrm{(2)}$ $f$ belongs to $K_P$ if and only if $f(X)>0$ for all $X\in P$, and the following  condition $\mathrm{(**)}$ holds: for all $i = 0, \ldots, n$, when $c_1 + \cdots + c_i  \leq k \leq c_1 + \cdots + c_i + c_{i+1} - 1$:

\begin{itemize}
    \item If $i$ is even, then $f(X_k) < f(X_{k+1})$;
    \item If $i$ is odd, then $f(X_k) > f(X_{k+1})$.
\end{itemize} More intuitively, condition $\mathrm{(**)}$ can also be written in the following form:
$$f(X_0)<\cdots < f(X_{c_1})> f(X_{c_1+1})> \cdots > f(X_{c_1+c_2})< \cdots > \ldots f(X_{r-1}) \sim f(X_r).$$
Here, $\sim$ is $< $ if $n$ is odd and $\sim$ is $> $ if $n$ is even.
\end{Lemma}

Henceforth, we shall use $\lambda(P)$ (or simply $\lambda$) to represent the maximum value among $c_i$ for $1 \leq i \leq n$. Additionally, we set $b_0 = 0$, and $b_i = c_1 + \ldots + c_i$ for $i = 1, \ldots, r$. Recall that for a function $f\in C_P$, the degree of $f$, denoted by $\deg(f)$, is defined as $f(P_i)$, where $i$ can be any number within $[0,r]$.
   The following lemma allows us to construct functions in $K_P$ satisfying certain properties conveniently.

\begin{Lemma} \label{construction} Let $P$  be a basic poset of type $(c_1,\ldots,c_n)$. Given $0\leq s<t\leq n$ and a function $f$ from $Q:=\{X_{i}\:\; b_s\leq i\leq b_t\}$ to $\mathbb{N}$ such that for any $j\in \{s,t\}$, $f(X_{b_j})=1$ if $j$ is even and $f(X_{b_j})\geq \lambda+1$ otherwise, and that the monotonicity of $f$ at $Q$ meets the requirements of condition $\mathrm{(**)}$ in Lemma~\ref{cri}. Then $f$ can be extended a function in  $K_P$ of degree $d$, where $d$ is any given number greater than or equal to $\max\{f(X)\:\; X\in Q, \quad \lambda+1\}+1$.
\end{Lemma}
\begin{proof} We extend the definition of $f$ in the following way:  For $0\leq i<s$ or $t\leq i\leq n$,
\begin{itemize}
\item If $i$ is even,  we define $f(X_{b_i}) = 1, f(X_{b_i+1})=2, \ldots, f(X_{b_{i+1}-1})=c_i$.
\item If $i$ is odd,  we define $f(X_{b_i}) = \lambda+1$ when $i\neq t$. It is worth noting  that if $i=t$, we do not need to define $f(X_{b_i})$.  Furthermore, we specify the values of $f(X_{b_i+1}), \ldots, f(X_{b_{i+1}-1})$ in such a way that they form a strictly decreasing sequence within the interval  $[2,\lambda]$. As an  example,  such a sequence could be $\lambda,\lambda-1,\ldots, \lambda-c_i+2$.
    \item For all $0 \leq i \leq r$, we define $f(Y_i) = d - f(X_i)$.
\end{itemize}
Note that if either $c_{i+1}=1$ or $i=n$, we do not need to define $f(X_{b_i+1}), \ldots, f(X_{b_{i+1}-1})$.

 By applying Lemma~\ref{cri}, it can be easily verified that $f$ indeed belongs to $K_P$ and has a degree of $d$, thereby completing the proof.
\end{proof}
We now evaluate  the $a$-invariant of the chain algebra associated with a basic poset.
\begin{Lemma} \label{a} Let $P$  be a basic poset of type $(c_1,\ldots,c_n)$. Then $a(\mathbb{K}[C_P])=-\lambda-2$.

\begin{proof}  We first show that $\deg(f)\geq \lambda+2$ for all  $f\in K_P$. Let $f\in K_P$. There exists $1\leq i\leq n$ such that $c_i=\lambda$. We may assume that $i$ is even w.l.o.g. Set $b_0=0$ and $b_j=c_1+\cdots+c_j$ for $j=1,\ldots,n$. Then $$1\leq   f(X_{b_i})<\cdots<f(X_{b_i+c_i}).$$
 This implies $f(X_{b_i+c_i})\geq \lambda+1$ and so $\deg(f)=f(X_{b_i+c_i})+f(Y_{b_i+c_i})\geq \lambda+2$.

By Lemma~\ref{construction}, there is a function in $K_P$, say $g$, such that $g(X_{b_i+j})=1+j$ for $j=0,\ldots,c_i=\lambda$ with degree $\lambda+2$. This completes the proof.
\end{proof}
\end{Lemma}

Based on this result, we can compute the regularity of the chain algebra of all pure posets of width 2. First, since the chain algebra $\mathbb{K}[C_P]$ is standard graded, its Hilbert series of  can be expressed as a rational function $\frac{Q(t)}{(1-t)^d}$, where $d$ is equal to the Krull dimension of $\mathbb{K}[C_P]$. Moreover, because $\mathbb{K}[C_P]$ is Cohen-Macaulay, its regularity  is equal to the degree of $Q(t)$. Consequently, we have $\mathrm{reg}(\mathbb{K}[C_P]) = d + a(\mathbb{K}[C_P])$.
The dimension $d$ can be determined using Theorem~\ref{krull}.  As stated in Subsection~\ref{ordinal}, the $a$-invariant of the chain algebra of the ordinal sum of two posets $P_1$ and $P_2$ is given by $
a(\mathbb{K}[C_{P_1 \oplus P_2}]) = \min\{a(\mathbb{K}[C_{P_1}]), a(\mathbb{K}[C_{P_2}])\}.$

Let $P$ be a pure poset of width $2$. Then there exists a poset $\overline{P}$ such that $\overline{P}$ has the same chain algebra as $P$, and $\overline{P}$ can be expressed as the ordinal sum of basic posets and posets that have rank $0$ and width $2$. For a poset with rank $0$ and width $2$, its chain algebra is a polynomial ring in two variables. In particular, it is Gorenstein with an $a$-invariant of $-2$. Thus, the $a$-invariant and hence the regularity of $\mathbb{K}[C_P]$ can be computed using these observations and Lemma~\ref{a}.

Finally, we will characterize when $\mathbb{K}[C_P]$ is Gorenstein or nearly Gorenstein whenever $P$ is a pure poset of width 2. Note that for any $f\in C_P$, $\deg(f)$ is equal to $f(P_i)$ for any $i\in [0,r]$.

\begin{Lemma} \label{level} Let $P$ be a basic poset of type   $(c_1,\ldots,c_n)$.  Then the following statements are equivalent:
\begin{enumerate}
  \item $\mathbb{K}[C_P]$ is pseudo-Gorenstein.
  \item There exists a unique function belonging to $K_P$ with degree $\lambda+2$.
  \item It holds that $c_i\in \{1,\lambda\}$ for all $i=1,\ldots, n$, and  that if $c_i=1$ then $1< i< n$ and $c_{i-1}=c_{i+1}=\lambda$.
\end{enumerate}
\end{Lemma}
\begin{proof}  $(1)\Leftrightarrow (2)$  By Lemma~\ref{a}, we have $a(\mathbb{K}[C_P])=-\lambda-2$. Due to this fact, the equivalence of (1) and (2) follows by the definition of pseudo-Gorensteinness.

$(3)\Rightarrow (2)$ Let $f$ be an arbitrary function in $K_P$ with $\deg(f) = \lambda + 2$. To establish the result, it suffices to prove that the values of $f(X)$ for all $X\in P$ are uniquely determined by the sequence $c_1,\ldots,c_n$. We proceed through the following  steps:

\textbf{Step 1}: Given that $c_1 = \lambda$, we invoke Lemma~\ref{cri} to conclude:
\[
1 \leq f(X_0) < f(X_1) < \cdots < f(X_{\lambda}) \leq \lambda + 1.
\]
Consequently, it must be that $f(X_j) = j + 1$ for $j = 0, \ldots, \lambda$.

\textbf{Step 2}:
Consider now the case where $c_{i+1} = \lambda$ for some $i$. Utilizing a similar argument as in Step 1, we deduce:
\begin{itemize}
\item If $i$ is even, then $f(X_{b_i + j}) = j + 1$ for $j = 0, \ldots, \lambda$.
\item If $i$ is odd, then $f(X_{b_i + j}) = \lambda + 1 - j$ for $j = 0, \ldots, \lambda$.
\end{itemize}
In particular, for each $i$ such that $c_{i+1} = \lambda$, the values of $f$ on the subset $$\{X_{b_i}, X_{b_i + 1}, \ldots, X_{b_i + c_{i+1}-1}, X_{b_{i+1}}\}$$ of $P$ are uniquely specified.

\textbf{Step 3}:
For indices $i$ where $c_{i+1} = 1$, note that $c_i = c_{i+2} = \lambda$. By the results of Step 2, the values of $f$ at $X_{b_i}$ and $X_{b_{i+1}}$ are already uniquely determined. Therefore, the values of $f(X_i)$ for $i=0,\ldots,r$ are uniquely determined.  Since $f(X_i) + f(Y_i) = \lambda + 2$,  it follows that the value of $f(Y_i)$ is also uniquely determined. In conclusion, the values of $f(X)$ for all $X\in P$ are uniquely specified, completing the proof.

$(2)\Rightarrow (3)$ Suppose on the contrary that there is $1\leq i\leq n$ such that $1<c_i<\lambda$. If $i$ is odd, we define functions $f$ and $g$ on the subset $\{X_{b_{i-1}}, X_{b_{i-1}+1} \ldots, X_{b_{i-1}+c_i-1}, X_{b_i}\}$ of $P$ by $f(X_{b_{i+1}+j})=j+1$ for $j=0,1,\ldots,c_i-1$ and $f(b_i)=\lambda+1$, and $g(X_{b_{i+1}})=1$, $g(X_{b_{i+1}+j})=j+2$ for $j=1,\ldots,c_i-1$ and $g(b_i)=\lambda+1$. By Lemma~\ref{construction}, both $f$ and $g$ could be extended to functions belonging to  $K_P$ with degree $\lambda+2$. This is a contradiction, implying that if $i$ is odd,  $c_i$ must be either 1 or $\lambda$. Analogously, we have $c_i$ must be either 1 or $\lambda$ if $i$ is even. Hence, we conclude that $c_i\in \{1,\lambda\}$ for all $i=1,\ldots,n$.

 Suppose that $c_1=1$.  Then, in view of Lemma~\ref{construction}, the functions $f$ and $g$ on the subset$\{X_0,X_1\}$ defined by $f(X_0)=1, f(X_1)=\lambda+1$ and $g(X_0)=2, g(X_1)=\lambda+1$ can be extended two distinct functions in $K_P$ with a degree of $\lambda+2$. This is a contradiction, thereby proving that $c_1=\lambda$. Similarly, we can show that $c_n=1$.

It remains  to show that there is no $2\leq i\leq n-2$ such that $c_i=c_{i+1}=1$. Suppose on the contrary  that $c_i=c_{i+1}=1$ for some $2\leq i\leq n-2$. If $i$ is odd, according to Lemma~\ref{construction}, there are functions $f$ and $g$ of degree $\lambda+2$ such that $f(X_{b_{i-1}})=1, f(X_{b_i})=2$ and $f(X_{b_{i+1}})=1$, and that $g(X_{b_{i-1}})=1, g(X_{b_i})=\lambda+1$ and $g(X_{b_{i+1}})=1$. This is impossible. Similarly, the case when  $i$ is even is also impossible.
\end{proof}

\begin{Proposition} \label{G2} Let $P$  be a basic poset of type $(c_1,\ldots,c_n)$. Then  the following statements are equivalent:

       $\mathrm{(1)}$ the chain algebra $\mathbb{K}[C_P]$ is nearly Gorenstein.

       $\mathrm{(2)}$  the chain algebra $\mathbb{K}[C_P]$ is Gorenstein.

       $\mathrm{(3)}$ it holds that $c_1=\cdots=c_n=\lambda.$

  \end{Proposition}
\begin{proof}

$(1)\Leftrightarrow (2)$  By Theorem~\ref{ind}, the chain polytope $D_P$ of $P$ is indecomposable, and by Proposition~\ref{IDP}, $\mathbb{K}[C_P]$ coincides with the Ehrhart ring of $D_P$. Now, the result follows from \cite[Theorem 32]{HKMM}.

(2) $\Rightarrow$ (3)  Suppose on the contrary that $c_1=\cdots=c_n=\lambda$ does not hold. By Lemma~\ref{level}, there is $1<k<n$ such that $c_k=1$ and $c_{k-1}=c_{k+1}=\lambda\geq 2$. We may assume $k$ is even without loss of generality. By Lemma~\ref{construction}, there is at least one function $\gamma \in K_P$ of degree $2\lambda+1$ satisfying: $$\gamma(X_{b_{k-2}})=1, \gamma(X_{b_{k-2}+1})=2, \ldots, \gamma(X_{b_{k-1}})=\lambda+1,\gamma(X_{b_{k}})=\lambda,$$  and $$\gamma(X_{b_k+1})=\lambda+1,\gamma(X_{b_k+1})=\lambda+2,\ldots,\gamma(X_{b_{k+1}})=2\lambda.$$
On the other hand, since $\mathbb{K}[C_P]$ is Gorenstein, there is an unique function in $K_P$, say $f$,   with degree $\lambda+2$ such that $K_P=f+C_P$. Note that $f(X_{b_{k-1}})=\lambda+1, f(X_{b_{k}})=1$, it follows  that $(\gamma-f)(X_{b_{k-1}})=0$ and  $(\gamma-f)(X_{b_{k}})=\lambda-1$. In view of Lemma~\ref{cri}, we have $\gamma-f$ does not belong to $C_P$ , a contradiction.

(3) $\Rightarrow$ (2) Let $f$ be  the unique function in $K_P$ with $\deg(f)=\lambda+2$. Then the sequence $f(X_0), f(X_1),\ldots, f(X_r)$ is as follows:
$$1,2,\ldots, \lambda,\lambda+1, \lambda,\lambda-1,\ldots, 1,2\ldots.$$
 From this sequence, we observe that $f(X_i)-f(X_{i+1})\in \{1,-1\}$ for all $i=0,\ldots,r-1$. To prove $\mathbb{K}[C_P]$ is Gorenstein, it suffices to show that for any $g\in K_P$, the function $g-f$ belongs to $C_P$. Take an arbitrary function $g$ from $K_P$. According to Lemma~\ref{cri},  we have for all $i=0,\ldots, r-1$, $f(X_i)$ and $g(X_i)$ have the same ordering. That is, $f(X_i)>f(X_{i+1})$ if and only if $g(X_i)>g(X_{i+1})$, and vice versa. Since $g(X_i)>0$ for all $i=0,\ldots,r$, it follows that $g(X_i)\geq f(X_i)$ for all $i$,   and so $h:=g-f$  is  a function from $P$ to $\mathbb{N}_0$.

Now, if $g(X_{i+1})>g(X_i)$, then  $g(X_{i+1})-g(X_i)\geq f(X_{i+1})-f(X_i)=1$, and so $h(X_{i+1})\geq h(X_i)$. Similarly,  if $g(X_{i+1})<g(X_i)$, then $g(X_{i})-g(X_{i+1})\geq f(X_{i})-f(X_{i+1})=1$, and so $h(X_{i+1})\leq h(X_i)$. Therefore, $h$ and $g$ share the same monotonicity, and so $h$ belongs to $C_P$ by Lemma~\ref{cri}, as desired.
\end{proof}

 Recalling  \cite[Theorem 2.7]{HMP}, we have if $R$ is the Segre product $R_1\#R_2\#\cdots\#R_n$ of Gorenstein standard graded $\mathbb{K}$-algebras $R_i$ with $a(R_i)=-a_i<0$ for $i=1,\ldots,n$, then $R$ is Gorenstein if and only if $a_1=\cdots=a_n$. Furthermore, $R$ is nearly Gorenstein  if and only if $|a_i-a_j|\leq 1$ for all $1\leq i,j\leq n$.

\begin{Example} \label{6.8}\em Let $P$ be a pure poset of rank $r$ such that $G_i(P)$ is a complete bipartite graph for all $i=0,\ldots,r-1$. Then $P$ is the ordinal sum of the subposets $P_0, P_1,\ldots,$ and  $P_r$.  Since $P_i$ is an anti-chain for each $i=0,\ldots,r$, $\mathbb{K}[C_{P_i}]$ is a polynomial ring with $|P_i|$ variables.  In particular, $\mathbb{K}[C_{P_i}]$ is  Gorenstein  with $a(\mathbb{K}[C_{P_i}])=-|P_i|$. Therefore, according to \cite[Theorem 2.7]{HMP}, we have $C_P$ is Gorenstein if and only if $|P_0|=\cdots=|P_r|$. Additionally,  $C_P$ is nearly  Gorenstein if and only if there is an integer $d\geq 1$ such that $|P_i|$ is equal to either $d$ or $d+1$ for each $i=0,\ldots,r$.
\end{Example}We are now ready to present the main result of the section.

\begin{Theorem} Let $P$ be a finite pure poset with  width  2. Denote by $\overline{P}$ the poset derived from $P$ by eliminating the disconnected bipartite graphs $G_i(P)$. Suppose that $\overline{P}$ can be expressed as the ordinal sum of basic posets $P_1, \ldots, P_s$ and two-element anti-chains $Q_1, \ldots, Q_t$. Furthermore, assume that  $P_i$ has type $(c_{i,1}, \ldots, c_{i,n_i})$ for $i = 1, \ldots, s$.

 $\mathrm{(1)}$  If $s=0$, then $\mathbb{K}[C_P]$ is always Gorenstein.

$\mathrm{(2)}$  If $t>0$ and $s>0$, then $\mathbb{K}[C_P]$  is not Gorenstein. Moreover, $\mathbb{K}[P]$ is nearly Gorenstein if and only if $c_{i,j}=1$ for all $i=1,\ldots,s, j=1,\ldots,n_i$.

$\mathrm{(3)}$ If $t=0$, then the following statements hold:

$\mathrm{(a)}$ $\mathbb{K}[C_P]$ is Gorenstein if  and only if there is a positive integer  $\lambda$ such that $c_{i,j}=\lambda$ for all $i=1,\ldots,s, j=1,\ldots,n_i$.

$\mathrm{(b)}$ $\mathbb{K}[C_P]$ is nearly  Gorenstein  if  and only if there is a positive integer  $\lambda$ and a subset $A\subseteq [1,s]$ such that $c_{i,j}=\lambda$ for $i\in A$ and $j=1,\ldots,n_i$, while $c_{i,j}=\lambda+1$ for $i\in [1,s]\setminus A$ and $j=1,\ldots,n_i$.
\end{Theorem}

\begin{proof} Note that for each $i$, $\mathbb{K}[C_{Q_i}]$ is a Gorenstein algebra with $a$-invariant  $-2$.  Therefore, these conclusions   are drawn by combining  Propositions~\ref{G2} with \cite[Theorem 2.7]{HMP}.
\end{proof}

We remark that conclusion (1) of the above theorem is actually a special case of Example~\ref{6.8}.

\vspace{3mm}

{\bf \noindent Acknowledgment:} This project is supported by NSFC (No. 11971338). We would like to express our sincere gratitude to Dr. Sora Miyashita, whose  comments on the initial  version of this paper brought us to the reference \cite{HKMM}, resulting in   the introduction and the exploration of the chain polytope in this  revised version.   Additionally, we sincerely thank the referee for his/her meticulous review and constructive suggestions, which have significantly enhanced the quality of this paper.

\vspace{3mm}
{\bf \noindent Conflict of interest:} The authors have no Conflict of interest directly relevant to the content of this article.

{\bf \noindent Data availability statement:} Data sharing not applicable to this article as no data sets were generated or analysed during the current study.

\end{document}